\documentclass[12pt]{amsart}
\usepackage{amssymb,latexsym}
\usepackage{enumerate}

\newtheorem{teo}{Theorem}[section]

\newtheorem{lema}[teo]{Lemma}
\newtheorem{prop}[teo]{Proposition}

\theoremstyle{definition}
\newtheorem*{xrem}{Remark}

\numberwithin{equation}{section}

\def\real{\mathbb{R}}
\def\supp{\mathop{\mbox{\normalfont supp}}\nolimits}

\begin{document}

\baselineskip=17pt

\title[Maximal functions associated to radial measures.]{On the lack of dimension free estimates in $L^p$ for maximal functions associated to radial measures.}

\author[A. Criado]{Alberto Criado}
\address{Alberto Criado\\ Departamento de Matem\'aticas, Universidad Aut\'onoma de Madrid\\ 28049 Madrid, Spain}
\email{alberto.criado@uam.es}

\date{}

\begin{abstract}
In a recent article J. Aldaz proved that the weak $L^1$ bounds for the centered maximal operator associated to finite radial measures cannot be taken independently with respect to the dimension. We show that at least for small $p$ near to 1 the same result holds for the $L^p$ bounds of such measures with decreasing densities. We also give some concrete examples, that include Gaussian measure, where better estimates with respect to the general case are obtained.
\end{abstract}


\maketitle

\section{Introduction and statement of the main results.}

Consider a Borel measure $\mu$ on $\real^n$. For any $g\in L^1_{loc}(\real^n)$ we define the associated centered maximal function on balls as:
$$
M_\mu g(x)=\sup_{R>0} \frac1{\mu(B(x,R))}\int_{B(x,R)} \left|g(y)\right|\,d\mu(y),
$$
where $B(x,R)$ is the ball with respect to certain norm of radius $R$, centered at x.
When $\mu=m_n$, i.e. the Lebesgue measure, the behavior of maximal functions has been studied by various authors. E.M. Stein proved that $M_{m_n}$ is bounded on $L^p$ for $p>1$ with a constant that can be taken independent of the dimension (see \cite{Stein1}, \cite{Stein2}, and also \cite{Stein3}). J. Bourgain and A. Carbery extended this result in the range $p>\frac32$ to the maximal function associated with balls given by arbitrary norms of $\real^d$ (see \cite{Bourgain1}, \cite{Bourgain2}, \cite{Bourgain3} and \cite{Carbery}), and D. M\"uller \cite{Muller} showed that if we restrict ourselves to the balls resulting from the $l^q$ norms, there are uniform bounds in dimension, for every $p>1$.

\bigskip

For the $L^1$ weak type bounds, in a joint work, E.M. Stein and J.O. Str\"omberg
\cite{SteinStromberg} proved that the operator norm of $M_{m^n}$ grows at most
like $\mathcal O(n\log n)$, when we consider arbitrary balls; and at most like $\mathcal
O(n)$ in the special case where $B$ is the Euclidean ball. In a recent article by J.M. Aldaz \cite{Aldaz2}, it was proven that the weak type bounds for the maximal functions associated to cubes grow to infinity with the dimension. An explicit lower bound for the growth of the constants was obtained in \cite{CriadoSoria}. This has been further improved in \cite{Aubrun}.

\bigskip

We can formulate the same problem for the existence of dimension free bounds, in the situation where $\mu$ is a finite rotational invariant measure. We will only consider the case where $B$ is the Euclidean ball. In \cite{MenarguezSoria2} it was proved that when $\mu$ is the Lebesgue measure, $M_\mu$ is a weakly bounded operator on $L^1_{rad}(\mu)$ with a constant that can be taken independent of the dimension. The proof also applies in the case of a radially increasing measure $\mu$. In \cite{Aldaz} it is shown that whenever the measure $\mu$ is radial and finite, the best constant $C_{1,\mu}$ in the weak $L^1(\mu)$ inequality for $M_\mu$ grows exponentially to infinity with the dimension, even when we restrict ourselves to radial functions. In this work we show that a similar result holds for the best constants $C_{\mu,p}$ of the $L^p(\mu)$ inequalities of $M_\mu$, even if restricting the action to radial functions. This is the content of our main result:

\bigskip

\begin{teo}\label{theorem} There exists $p_0>1$ such that if $1\leq p<p_0$, there is an $\alpha>1$ so that for every $n\in\mathbb N$ and every finite Borel measure $\mu$ on $\real^n$, with a radially decreasing density, one has
$$
C_{\mu,p} > c \alpha^n,
$$
where $c$ is an absolute constant, independent of the dimension.
\end{teo}

\bigskip

This means that E.M. Stein's result of dimension free $L^p$ bounds for maximal functions associated to Euclidean balls is not extendable to the context of radial finite measures.

\bigskip

For any $r>0$ we and $x\in\real^n$ denote by $B(x,r)$ the ball centered at $x$ with radius $r$. We will write $B_r$ to denote $B(0,r)$. By $\mathbf \xi$ we will denote a unit vector of $\real^n$, arbitrary since our setting is rotational invariant. The proof of Theorem \ref{theorem} relies on the following proposition:

\bigskip

\begin{prop}\label{proposicion} Let $\mu$ be a finite Borel measure on $\real^n$ with a radially decreasing density. Given $R>r>0$ two positive radii, and $\xi$ a unitary vector, write $\tilde B = B(R\mathbf \xi, R+r)$. Then,
\begin{equation}\label{estimacion.constante}
C_{\mu,p}\geq \frac{\mu(B_R)}{\mu(\tilde B)} \left(\frac{\mu(B_r)}{\mu(B_R)}\right)^{\frac{p-1}p}=:T_{\mu,p}(R,r).
\end{equation}
\end{prop}

\bigskip

In section \ref{proof}, Theorem \ref{theorem} and Proposition \ref{proposicion} are proven. In sections \ref{gaussian} and \ref{unit.ball} we study two particular cases, the Gaussian measure and Lebesgue measure restricted to the unit ball, where it is possible to develop explicit computations of $T_{\mu,p}(R,r)$. In both cases, we will obtain better exponents $p$ than in Theorem \ref{theorem} and will show too that the estimates are somehow optimal, in the sense that the argument given by Proposition \ref{proposicion} cannot be much extended.

\bigskip

We have learned from J.M. Aldaz that he is also working in this problem and that in particular, he has independently found similar results to those in our Theorem \ref{theorem}.

\bigskip

\noindent\textit{Some notation and preliminary facts:} For any $\delta\in[-1,1]$, we define the cone with $\delta$ aperture as $E_\delta :=\left\{x\in\real^n :\,x\cdot\xi\geq\delta|x|\right\}$. Let $\omega_{n-1}$ be the measure of $\mathbb S^{n-1}$ with respect to the surface measure induced by the Lebesgue measure on $\real^n$. It is a known fact that $\omega_{n-1}=n|B(0,1)|=n\pi^\frac n2/\Gamma(\frac n2+1)$. We shall use repetitively that
\begin{equation}\label{cociente.esferas}
\frac1{\sqrt\pi}\frac{n-1}n\leq\frac{\omega_{n-2}}{\omega_{n-1}}\leq \frac{n-1}{\sqrt{2\pi}}\left(1+\frac1n\right)^\frac12,
\end{equation}
where the first inequality follows immediately from the definition and the second one is a consequence of the log-convexity of $\Gamma$ (see \cite{Aldaz} or \cite{Webster}).

\bigskip

\section{Proofs of the theorem and the proposition.}\label{proof}

\bigskip

\begin{proof}[Proof of Proposition \ref{proposicion}]
Consider the function $g(x)=\frac1{\mu(B_r)}\chi_{B_r}(x)$. By the Tchebychev inequality a strong $L^p$ bound implies a weak one:
\begin{eqnarray}\label{chebychev}
\mu\left(\left\{x:\,Mg(x)>\frac 1{\mu(\tilde B)}\right\}\right) &\leq& \frac1{\mu(\tilde B)^p} \int Mg(x)^p\,d\mu(x)\nonumber\\ &\leq& \frac {C_{\mu,p}^p}{\mu(\tilde B)^p} \int g(x)^p\,d\mu(x).
\end{eqnarray}
We claim that
\begin{equation}\label{bola.en.conjunto.nivel}
B_R \subset \left\{x:\,Mg(x)>\frac 1{\mu(\tilde B)}\right\},
\end{equation}
so, rearranging (\ref{chebychev}) we obtain:
\begin{equation}
C_{\mu,p}\geq \frac{\mu(B_R)}{\mu(\tilde B)} \left(\frac{\mu(B_r)}{\mu(B_R)}\right)^{\frac{p-1}p}.
\end{equation}
To prove (\ref{bola.en.conjunto.nivel}), take any $x$ in $B_R$, then
\begin{eqnarray*}
Mg(x)&\geq& \frac1{\mu(B(x,|x|+r))}\int_{B(x,|x|+r))}g(x)\,d\mu(x)\\ &\geq& \frac1{\mu\left(B\left(R\frac x{|x|},R+r\right)\right)}= \frac1{\mu(\tilde B)}.
\end{eqnarray*}
Here we have used the rotation invariance of $\mu$.
\end{proof}

\bigskip

\begin{proof}[Proof of theorem \ref{theorem}] We are going to bound $T_{\mu,p}(R,r)$ from below. Write $d\mu(x)=f(|x|)\,dx$. Set $\lambda:= \frac rR$. By our hypothesis
the function $f$ is decreasing, so
\begin{equation}\label{compara.bolas.pequeñas}
\frac{\mu(B_R)}{\mu(B_r)} \leq 1 +\frac{\int_r^R f(r) s^{n-1}\,ds}{\int_0^r f(r)s^{n-1}\,ds} = \lambda^{-n}.
\end{equation}

\bigskip

 Now we compare the $\mu$-measures of $\tilde B$ and $B_R$. Following Aldaz, we split $\tilde B$ into two disjoint pieces,
\begin{equation}\label{partir.bola.grande}
\mu(\tilde B)=\mu(\tilde B\cap B_R)+\mu(\tilde B\setminus B_R).
\end{equation}
We denote by $\beta_0(r)$ the angle between $\mathbf \xi$ and the segment that connects the origin with any point in $\partial \tilde B\cap \partial \tilde B_r$. For notational simplicity, call $\beta_0 := \beta_0(R)$.

\bigskip

Using that $\tilde B\cap B_R \subset B(R \cos \beta_0 \mathbf \xi, R \sin \beta_0)$ and that $\mu$ is a radially decreasing measure,
\begin{equation}\label{estimacion.interseccion}
\mu(\tilde B\cap B_R)\leq \mu(B(R \cos \beta_0 \mathbf \xi, R \sin \beta_0)) \leq \mu(B_{R \sin \beta_0}).
\end{equation}
Given that $R\mapsto \frac{\mu(B_{R\sin \beta_0})}{\mu(B_R)}$ is a continuous function that tends to 1 when $R\rightarrow\infty$ and by Lebesgue differentiation theorem to $\sin^n \beta_0$ when $R\rightarrow0$, it is possible to find an $R$ such that:
\begin{equation}\label{cociente.bolas.exacto}
\mu(B_{R\sin \beta_0}) =(\sin\beta)^{nk} \,\mu(B_R),
\end{equation}
where $k\in(0,1)$ depends on $\lambda$ and will be chosen later.

\bigskip

The cosine theorem applied to the triangle whose vertices are the origin, $R\mathbf \xi$, and a point in $\partial \tilde B\cap \partial B_R$ yields
$$
\cos \beta_0= 1-\frac{(1+\lambda)^2}2.
$$
By choosing $\lambda< \sqrt 2-1$, we make $\cos \beta_0 >0$. By integrating in spherical coordinates,
\begin{eqnarray}\label{cuenta.integral}
\mu(\tilde B \setminus B_R) &=& \omega_{n-2} \int_R^{2R+r} \int_0^{\beta_0(r)} (\sin\beta)^{n-2}\,d\beta\, f(s)s^{n-1}\,ds \nonumber\\ &\leq& \frac{\omega_{n-2}}{\cos \beta_0} \int_R^{2R+r} \int_0^{\beta_0} (\sin\beta)^{n-2} \cos\beta\,d\beta\, f(s)s^{n-1}\,ds \nonumber\\ &=& \frac{\omega_{n-2}}{\cos\beta_0}\frac{(\sin\beta_0)^{n-1}}{n-1} \int_R^{2R+r} f(s)s^{n-1}\,ds \nonumber\\ &\leq& \frac1{\sqrt{\pi}\,\sin\beta_0\,\cos\beta_0} (\sin\beta)^n \,\mu(B_{2R+r}\setminus B_R),
\end{eqnarray}
where for the last inequality we used (\ref{cociente.esferas}). As $\sin \beta_0 <1$, it holds that $(\sin\beta_0)^{-l}\,\,R>2R+r$ for a big enough positive integer $l$. For example the choice
\begin{equation}
l=\left\lceil-\frac{\log (2+\lambda)}{\log \sin \beta_0}\right\rceil,
\end{equation}
will do; so if we assume $R$ to be the maximal $R>0$ for which (\ref{cociente.bolas.exacto}) is satisfied, then
\begin{equation}\label{estimacion.bola.exterior}
\mu(B_{2R+r})\leq \mu(B_{(\sin\beta_0)^{-l}\, R}) \leq (\sin\beta_0)^{-ln} \mu(B_R).
\end{equation}

\bigskip

Putting together (\ref{partir.bola.grande}), (\ref{cociente.bolas.exacto}), (\ref{cuenta.integral}) and (\ref{estimacion.bola.exterior}) we get,
\begin{equation}\label{exponentes.distintos}
\frac{\mu(B_R)}{\mu(\tilde B)} \geq \frac1{Q\left(\sin\beta_0\right)^{n(1-lk)} + \left(\sin \beta_0\right)^{nk}},
\end{equation}
where $Q=(\sqrt{\pi}\,\sin\beta_0\,\cos\beta_0)^{-1}$. The right hand side of (\ref{exponentes.distintos}) attains its maximal growth with respect to $n$ when the two terms in the denominator are of the same exponential size, so we need that $1-lk=k$; this fixes $k=1/(1+l)$. By (\ref{cociente.bolas.exacto}) and the observation before (\ref{estimacion.bola.exterior}) it also determines $R$, and, as $\lambda$ was previously chosen, $r$ gets fixed too.

\bigskip

Using (\ref{exponentes.distintos}) with $1-lk=k$ and (\ref{compara.bolas.pequeñas}) on (\ref{estimacion.constante}) we obtain:
$$
C_{\mu,p} \geq T_{\mu,p}(R,r) \geq \frac 1{Q+1} \left(\frac {\lambda^{\frac{p-1}p}}{\sin^k\beta_0}\right)^n.
$$
It only remains to observe that, although $R$ (and consequently $r$) can change with the dimension (see the remark below), it is possible to choose an universal $\lambda$, so that neither $k,l$ and $\beta_0$ depend on $n$, and so,
$$
\alpha = \frac {\lambda^{\frac{p-1}p}}{\sin^k\beta_0}>1,
$$
will hold for $p$ close enough to 1, namely when
$$
p< \frac{\log \lambda}{\log\left(\frac \lambda{\sin^k\beta_0}\right)}.
$$
So just take
$$
p_0=\sup_{\lambda\in (0,\sqrt2-1)}\frac{\log \lambda}{\log\left(\frac \lambda{\sin^k\beta_0}\right)}.
$$
The analytic computation of $p_0$ is a rather complicated calculation. By a numerical estimate via MatLab we obtained $p_0\approx1.005274$.
\end{proof}

\bigskip

\begin{xrem} It is interesting to make the following observations about the radii chosen above. Let $f:\real_+\rightarrow\real_+$ be a decreasing function such that all the measures $d\mu_n=f(|x|)\,dx$ are finite on $\real^n$. Fixing $\lambda$ and taking $\beta_0$ as in the previous proof, we define $R_n$ as the maximal radius for which (\ref{cociente.bolas.exacto}) holds for $\mu_n$ in $\real^n$. This radius does not shrink to $0$ as the dimension grows. Given $R_0>0$ such that $f(R_0)>0$, note that since $f$ is decreasing
$$
\frac{\mu(B_{R_0\sin \beta_0})}{\mu(B_{R_0})} \leq \frac{f(0)}{f(R_0)}(\sin\beta_0)^n,
$$
which is smaller than $(\sin\beta)^{nk}$ for large $n$, thus for (\ref{cociente.bolas.exacto}) it is necessary to take $R_n>R_0$. So if $f$ has compact support then $\liminf_{n\rightarrow\infty} R_n\geq \max \supp f$, and if the support of $f$ is unbounded $\lim_{n\rightarrow\infty}R_n=\infty$. By the definition of $R_n$ and the decreasing property of $f$,
$$
(\sin\beta_0)^{nk}=\frac{\mu(B_{R_n\sin \beta_0})}{\mu(B_{R_n})} \leq \frac{f(0)}{f(R_n)}(\sin\beta_0)^n,
$$
so we obtain the exponential decaying $f(R_n)\leq f(0)(\sin\beta_0)^{n(1-k)}$.
\end{xrem}
\bigskip

\section{The Gaussian measure.}\label{gaussian}

\bigskip

In the case of the Gaussian measure $d\mu(x)=e^{-\pi|x|^2}\,dx$ it is possible to make a better estimate of the quantities implied in $T_{\mu,p}(R,r)$. We will obtain unboundedness with respect to the dimension of the $L^p(\mu)$ norms with bigger $p$ than those in Theorem \ref{theorem} and we will show that, by means of Proposition \ref{proposicion}, no much bigger exponents $p$ can be reached.

\bigskip

\begin{prop}\label{proposicion.gaussiana}
Let $R_n:=\sqrt{\frac{n-1}{2\pi}}$. There exist $p_1>p_0>1$ with approximated value $p_0\approx1.011871$ and $p_1\approx1.049427$ such that
\begin{enumerate}[(i)]
\item for every $p<p_0$ there exists an $\alpha>1$ (only depending on $p$) such that $C_{\mu,p}\geq T_{\mu,p}(R,r)\geq \alpha^n$ for some $0<r<R< R_n$ with $\lambda=\frac Rr<\sqrt2-1$.
\item for every $p>p_1$ given any choice $0<r<R\leq R_n$ there exists an $\alpha<1$ (only depending on $p$) such that $T_{\mu,p}(R,r)\leq \sqrt\pi n \alpha^n$.
\end{enumerate}
\end{prop}

\bigskip

Let us comment some aspects on the behavior of Gaussian measures. By an integration in polar coordinates the Gaussian measure of a centered ball $B_\rho$ can be written as
$$
\mu(B_\rho)=\omega_{n-1}\int_0^\rho e^{-\pi s^2} s^{n-1}\,ds.
$$
It is an elementary calculus exercise to realize that the function $h_n(s)=e^{-\pi s^2} s^{n-1}$ increases from $s=0$ until the point $s=R_n:=\sqrt{\frac{n-1}{2\pi}}$ where it attains its maximum and from this point on, to infinity it decreases. Moreover $h$ is concave in the interval $(R_n^-,R_n^+)$, and convex in the two complementary intervals of $(0,\infty)$, where $R_n^\pm = \sqrt{\frac{2n-1\pm \sqrt{8n-7}}{4\pi}}$. This gives us the following estimates for $B_\rho$ when $\rho\leq R_n$.

\bigskip

\begin{lema}\label{estimacion.gaussiana} Let $0<\rho<R_n$, one has
$\omega_{n-1} e^{-\pi \rho^2}\frac{\rho^n}n\leq \mu(B_\rho)\leq \omega_{n-1} e^{-\pi \rho^2}\rho^n$.
\end{lema}

\bigskip

It is also easy to check that almost all the mass of $\mu$ is supported in the ball $B_{R_n}$, as the following lemma asserts.

\bigskip

\begin{lema}\label{casi.toda.la.masa}
One has that $\mu(B_{R_n})\geq 1-\frac2{\sqrt\pi\sqrt{n-1}}$.
\end{lema}

\bigskip

\begin{xrem} Before proving the Proposition and the Lemmas let us justify that the only interesting case is the one where $0<r<R<R_n$ and $\lambda<\sqrt 2 -1$. In view of Lemma \ref{casi.toda.la.masa} there is no point in considering large radii. To take $r>R_n$ makes no sense, since then all the measures of the balls involved in $T_{\mu,p}(R,r)$ are close to 1. In the case $r<R_n<R$, by Lemma \ref{casi.toda.la.masa}, we have that $\frac13\leq \mu(B_R) \leq 1$ for any $n$, so this means
$$
\frac1{3\mu(\tilde B)}\left(\mu(B_r)\right)^\frac{p-1}p\leq T_{\mu,p}(R,r)\leq\frac1{\mu(\tilde B)}\left(3\mu(B_r)\right)^\frac{p-1}p.
$$
Here it is clearly seen that increasing $R$ over $R_n$ only makes $\mu(\tilde B)$ bigger, which is of no use in order to bound $T_{\mu,p}(R,r)$ from below. If $\lambda\geq \sqrt 2-1$, then $\tilde B\supset E_0\cap B_R$, and as $\mu(E_0\cap B_R)=\frac12\mu(B_R)$ one has
$$
T_{\mu,p}(R,r)\leq 2.
$$
\end{xrem}

\bigskip

\begin{proof}[Proof of Proposition \ref{proposicion.gaussiana}]
Let us first demonstrate (\ref{i}). As in Theorem \ref{theorem} we will bound from below $T_{\mu,p}(R,r)=\frac{\mu(B_R)}{\mu(\tilde B)} \left(\frac{\mu(B_r)}{\mu(B_R)}\right)^\frac{p-1}p$.
We consider the aforementioned partition $\mu(\tilde B) =\mu(\tilde B\cap B_R)+\mu(\tilde B\setminus B_R)$. Following the same reasonings that led to (\ref{estimacion.interseccion}) together with Lemma \ref{estimacion.gaussiana} we obtain that,
\begin{equation}\label{interseccion.gaussiana}
\mu(\tilde B\cap B_R) \leq \mu(B_{R\sin\beta_0}) \leq \omega_{n-1}e^{-\pi R^2 \sin^2\beta_0}(R\sin\beta_0)^n.
\end{equation}
With the argument contained in (\ref{cuenta.integral}) and recalling that $h$ attains its maximum at the point $R_n$:
\begin{eqnarray}\label{restante.gaussiana}
\mu(\tilde B\setminus B_R) &\leq& \frac{(\sin\beta_0)^n}{\sqrt\pi \sin\beta_0\cos\beta_0}\omega_{n-1}\int_R^{2R+r}e^{-\pi s^2}s^{n-1}\,ds \nonumber\\&\leq& \frac{(\sin\beta_0)^n}{\sqrt\pi \sin\beta_0\cos\beta_0}\omega_{n-1} (R+r) e^{-\pi R_n^2} R_n^{n-1}.
\end{eqnarray}
We would like to find an $R$ such that both the righthand sides of (\ref{interseccion.gaussiana}) and (\ref{restante.gaussiana}) are of the same exponential size with respect to $n$. This leads us to the transcendental equation
$$
e^{-\pi R^2 \sin^2\beta_0}R^n=e^{-\pi R_n^2} R_n^{n-1}.
$$
We will take as $R$ the approximated solution $R=e^{-\frac12\cos^2\beta_0}R_n=\left(\frac{n-1}{2\pi e^{\cos^2\beta_0}}\right)^\frac12$.
Substituting in (\ref{interseccion.gaussiana}) and (\ref{restante.gaussiana}) we obtain
\begin{eqnarray*}
\mu(\tilde B\cap B_R) &\leq & \omega_{n-1}\sqrt e\, e^{-\frac n2(\sin^2\beta_0e^{-\cos^2\beta_0}+\cos^2\beta_0)} \left(\frac{n-1}{2\pi}\right)^\frac n2 (\sin\beta_0)^n,\\
\mu(\tilde B\setminus B_R) &\leq&
\omega_{n-1}\frac{\sqrt e\,(\sin\beta_0)^n}{\sqrt\pi \sin\beta_0\cos\beta_0}\, e^{-\frac n2} \left(\frac{n-1}{2\pi}\right)^\frac n2.
\end{eqnarray*}
As $\frac n2(\sin^2\beta_0e^{-\cos^2\beta_0}+\cos^2\beta_0)>\frac n2$ the right hand side of the first inequality dominates exponentially the one in the second inequality, so for large dimensions
\begin{equation}\label{bola.gaussiana.tilde}
\mu(\tilde B)\leq \omega_{n-1} 2\, e^{-\frac n2(\sin^2\beta_0e^{-\cos^2\beta_0}+\cos^2\beta_0)} \left(\frac{n-1}{2\pi}\right)^\frac n2 (\sin\beta_0)^n.
\end{equation}
Using Lemma \ref{estimacion.gaussiana} again we obtain
\begin{equation}\label{bola.gaussiana.por.debajo}
\mu(B_R)\leq \sqrt e\,\frac{\omega_{n-1}}n e^{-\frac n2 ^(e^{-\cos^2\beta_0}+\cos^2\beta_0)} \left(\frac{n-1}{2\pi}\right)^\frac n2,
\end{equation}
as well as
\begin{equation}\label{bola.gaussiana.cociente}
\frac{\mu(B_r)}{\mu(B_R)}\geq \frac{e^{-\pi r^2}\frac{r^n}n}{e^{-\pi R^2}R^n} =  \frac1n e^{\pi R^2(1-\lambda^2)}\lambda^n = e^{\frac{n-1}2e^{-\cos^2\beta_0}(1-\lambda^2)}\lambda^n.
\end{equation}
Now to estimate $T_{\mu,p}(R,r)$ put together (\ref{bola.gaussiana.tilde}), (\ref{bola.gaussiana.por.debajo}) and (\ref{bola.gaussiana.cociente}) to get
$$
T_{\mu,p}(R,r) \geq \frac 1n \frac{e^{-\frac n2\cos^2\beta_0\,e^{-\cos^2\beta_0}}}{(\sin\beta_0)^n} \left(e^{\frac n2e^{-\cos^2\beta_0}(1-\lambda^2)}\lambda^n\right)^\frac{p-1}p,
$$
which will grow to infinity with the dimension only if
$$
\frac{e^{-\frac 12\cos^2\beta_0\,e^{-\cos^2\beta_0}}}{\sin\beta_0} \left(e^{\frac 12e^{-\cos^2\beta_0}(1-\lambda^2)}\lambda\right)^\frac{p-1}p>1.
$$
This is equivalent with
$$
p<\frac{\log\left(\lambda e^{-\cos^2\beta_0(\sin^2\beta_0-\lambda^2)}\right)}{\log\frac{\lambda e^{-\cos^2\beta_0(\sin^2\beta_0-\lambda^2)}}{\sin\beta_0}}.
$$
So we can take
$$
p_0:=\sup_{0<\lambda<\sqrt2-1}\frac{\log\left(\lambda e^{-\cos^2\beta_0(\sin^2\beta_0-\lambda^2)}\right)}{\log\frac{\lambda e^{-\cos^2\beta_0(\sin^2\beta_0-\lambda^2)}}{\sin\beta_0}}.
$$
A numerical estimation via Matlab yields the approximative value $p_0\approx 1.011871$.

\bigskip

It remains to prove (\ref{ii}). Given that $R< R_n$ by Lemma \ref{estimacion.gaussiana},
\begin{equation}\label{cociente.facil}
\frac{\mu(B_r)}{\mu(B_R)} \leq \frac{e^{-\pi r^2}r^n}{e^{-\pi R^2}\frac{R^n}n} =  e^{\pi R^2(1-\lambda^2)}\lambda^n < e^{\frac{n-1}2(1-\lambda^2)}\lambda^n.
\end{equation}
The ball $\tilde B$ contains the part of the cone $E_{\cos\beta_0}$ included in $B_R$, therefore
\begin{equation}
\mu(\tilde B) \geq \mu(\tilde B\cap B_R) \geq \mu (E_{\cos \beta_0} \cap B_R),
\end{equation}
where $E_{\cos \beta_0}:=\left\{x : \frac{ x\cdot \xi}{|x|}>\cos \beta_0\right\}$. Integrating in spherical coordinates,
\begin{eqnarray*}
\mu(E_{\cos\beta_0} \cap B_R) &=& \int_0^R e^{-\pi s^2} \int_0^{\beta_0} \omega_{n-2} s^{n-2}(\sin\beta)^{n-2} \,d\beta\,s\,ds \\ &\geq& \omega_{n-2} \int_0^R f(s) \int_0^{\beta_0} (\sin\beta)^{n-2} \cos \beta \,d\beta \,f(s)s^{n-1}\,dr \\ &\geq& \frac1{\sqrt \pi n\,\sin \beta_0} \sin^n \beta_0 \,\mu(B_R),
\end{eqnarray*}
where for the last inequality we used (\ref{cociente.esferas}). Now one has
$$
T_{\mu,p}(R,r) \leq \sqrt\pi n\sin\beta_0 \,e^{\frac{\lambda^2-1}2\frac{p-1}p}\left(\frac{(e^\frac{1-\lambda^2}2\lambda)^{\frac{p-1}p}}{\sin \beta_0}\right)^n.
$$
The right hand side of the previous inequality tends to $0$ when $n\rightarrow \infty$ if
\begin{equation}\label{lambda.gamma}
\frac{(e^\frac{1-\lambda^2}2\lambda)^{\frac{p-1}p}}{\sin \beta_0}<1.
\end{equation}
Since $\frac{e^\frac{1-\lambda^2}2\lambda}{\sin\beta_0}<1$ for every $\lambda<\sqrt 2-1$, (\ref{lambda.gamma}) is equivalent with
$$
p>\frac{\log e^\frac{1-\lambda^2}2\lambda}{\log\frac{e^\frac{1-\lambda^2}2\lambda}{\sin\beta_0}}.
$$
so $T_{\mu,p}(R,r)$ shrinks to $0$ exponentially with the dimension for any $p$ greater than
$$
p_1 = \sup_{0<\lambda<\sqrt 2 -1} \frac{\log e^\frac{1-\lambda^2}2\lambda}{\log\frac{e^\frac{1-\lambda^2}2\lambda}{\sin\beta_0}}.
$$
A numerical estimate via MatLab yields $p_1\approx 1.049427$.
\end{proof}

\bigskip

\begin{proof}[Proof of Lemma \ref{estimacion.gaussiana}]
Integrating in radial coordinates we have
$$
\mu(B_\rho)=\omega_{n-1}\int_0^\rho e^{-\pi s^2}s^{n-1}\,ds.
$$
On the one hand, as $h$ is increasing it attains its maximal value in $[0,\rho]$ at the point $s=\rho$, so
$$
\int_0^\rho e^{-\pi s^2}s^{n-1}\,ds\leq \rho e^{-\pi \rho^2} \rho^{n-1}=e^{-\pi\rho^2};
$$
on the other hand
$$
\int_0^\rho e^{-\pi s^2}s^{n-1}\,ds \geq e^{-\pi \rho^2} \int_0^\rho s^{n-1}\,ds = e^{-\pi \rho^2} \frac{\rho^n}n.
$$
\end{proof}

\bigskip

\begin{proof}[Proof of Lemma \ref{casi.toda.la.masa}]
First take into account that,
\begin{eqnarray*}
\mu(B_{R_n})&=&1-\omega_{n-1}\int_{R_n}^\infty e^{-\pi s^2}s^{n-1}\,ds \geq 1 - \omega_{n-1}R_n^{n-2}\int_{R_n}^\infty e^{-\pi s^2}s\,ds \\&=& 1-\frac{\omega_{n-1}}{2\pi}R_n^{n-2}e^{-\pi R_n^2}.
\end{eqnarray*}
By the Stirling formula $\Gamma(t)=\sqrt\frac{2\pi}t\left(\frac te\right)^t\left(1+\mathcal{O}\left(\frac1t\right)\right)$,
$$
\omega_{n-1}=\frac{n\pi^\frac n2}{\Gamma\left(\frac n2+1\right)}\leq \frac{8e}{3\sqrt{2\pi}}\frac{(\pi e)^\frac n2}{\left(\frac n2+1\right)^{\frac{n-1}2}}.
$$
Now we are done, because
$$
\frac{\omega_{n-1}}{2\pi}R_n^{n-2}e^{-\pi R_n^2} \leq
\frac{2e}{\pi\sqrt{2\pi}}\frac{(\pi e)^\frac n2}{\left(\frac n2+1\right)^{\frac{n-1}2}}\left(\frac{n-1}{2\pi}\right)^{\frac{n-2}2}e^{-\frac{n-1}2}\leq \frac2{\sqrt\pi\sqrt{n-1}}.
$$
Note that $\left(\frac{n-1}{n+2}\right)^\frac{n-1}2$ tends to $e^{-\frac32}$ and for large $n$ can be bounded by $\frac32 e^{-\frac32}$.
\end{proof}

\bigskip

\section{Lebesgue measure restricted to the unit ball}\label{unit.ball}

\bigskip

If we consider the radial measure $d\mu(x) = \chi_{B_1}(x)\,dx$, there is a more direct way to estimate $T_{\mu,p}(R,r)$. We will obtain unboundedness of $C_{\mu,p}$ with respect to the dimension for larger $p$'s than in Theorem \ref{theorem}. The method is optimal in the sense that no bigger exponents $p$'s can be reached using Proposition \ref{proposicion}. This is the content of the next Proposition.

\bigskip

\begin{xrem} There is no point in considering the case $r\geq 1$ since then $T_{\mu,p}(R,r)=1$ trivially. If $r<1$, taking $R>1$ only increases $\mu(\tilde B)$, which makes $T_{\mu,p}(R,r)$ smaller. So we will concentrate on the situation where $0<r<R\leq 1$.
\end{xrem}

\bigskip

\begin{prop}
There exists a $p_0>1$ with approximated value $p_0\approx1.03946$ such that
\begin{enumerate}[(i)]
\item\label{i} for every $p<p_0$ there exists an $\alpha>1$ (only depending on $p$) such that $C_{\mu,p}\geq T_{\mu,p}(1,r)\geq \alpha^n$ for some $r<\sqrt2-1$,
\item \label{ii}for every $p>p_0$ given a choice $0<r<R\leq 1$ there exists an $\alpha<1$ (depending on $p,r,R$) such that $T_{\mu,p}(R,r)\leq \sqrt\pi n \alpha^n$.
\end{enumerate}
\end{prop}

\bigskip

\begin{proof}
We shall estimate
$$
T_{\mu,p}(R,r)=\frac{\mu(B_R)}{\mu(\tilde B)} \left(\frac{\mu(B_r)}{\mu(B_R)}\right)^{\frac{p-1}p}.
$$
Calling $\lambda=r/R$ again, $\frac{\mu(B_r)}{\mu(B_R)}=\lambda^n$.

\bigskip

Let us denote by $\beta_0$ the angle determined by $\mathbf \xi$ and a segment that connects the origin with any point in $\partial\tilde B \cap \partial\tilde B_1$. By the cosine theorem on this triangle, $\cos\beta_0=1-\frac{R^2(1+\lambda)^2}2$. One has the inclusion $E_{\cos\beta_0}\cap B_1 \subset \tilde B\cap B_1$, and taking $R<\frac{\sqrt2}{1+\lambda}$ makes $\cos\beta_0>0$, so one can integrate in spherical coordinates to obtain,
\begin{eqnarray*}
\mu(\tilde B) &\geq& |E_{\cos \beta_0}\cap B_1| =\int_0^1\int_0^{\beta_0}\omega_{n-2}(s\sin\beta_0)^{n-2}s\,d\beta\,ds \\&\geq& \omega_{n-2}\int_0^{\beta_0}(\sin\beta_0)^{n-2}\cos\beta\,d\beta\int_0^1s^{n-1}\,ds\\&=& \frac{\omega_{n-2}(\sin\beta_0)^{n-1}}{n-1}\frac1{\omega_{n-1}}|B_1| \geq \frac1{\sqrt\pi n}(\sin\beta)^n |B_1|,
\end{eqnarray*}
where the last inequality follows from (\ref{cociente.esferas}). On the other hand, as $\tilde B\cap B_1\subset B(\cos\beta_0 \mathbf \xi, \sin\beta_0)$, and $\mu$ is radially decreasing,
$$
\mu(\tilde B) \leq \mu(B(\cos\beta_0 \mathbf \xi, \sin\beta_0)) \leq |B_{\sin\beta_0}| = (\sin\beta)^n\,|B_1|.
$$
Thus,
\begin{equation}\label{cota.arriba.abajo}
 \left(\frac{R\lambda^\frac{p-1}p}{\sin\beta_0}\right)^n\leq T_{\mu,p}(R,r)\leq \sqrt\pi n\left(\frac{R\lambda^\frac{p-1}p}{\sin\beta_0}\right)^n.
\end{equation}

\bigskip

We will first concentrate on (\ref{i}). The condition $R>\sin\beta_0$ is necessary for the three quantities in (\ref{cota.arriba.abajo}) to grow with the dimension. This is equivalent with $R>2\sqrt{(1+\lambda)^{-2}-(1+\lambda)^{-4}}$, which is compatible with our previous assumption, $R<\frac{\sqrt2}{1+\lambda}$, only if $0<\lambda<\sqrt 2-1$. But for this range of $\lambda$, as $R\leq 1$, the condition $R<\frac{\sqrt2}{1+\lambda}$ is not a restriction anymore. Taking $R=1$, implies $r=\lambda$ and (\ref{cota.arriba.abajo}) becomes
\begin{equation}\label{cota.arriba.abajo1}
\left(\frac{\lambda^\frac{p-1}p}{\sin\beta_0}\right)^n\leq T_{\mu,p}(1,\lambda)\leq \sqrt\pi n\left(\frac{\lambda^\frac{p-1}p}{\sin\beta_0}\right)^n.
\end{equation}
This quantities will tend to infinity as $n\rightarrow\infty$ only if $\alpha_{p,\lambda}:=\frac{\lambda^\frac{p-1}p}{\sin\beta_0}>1$ and this is equivalent with $p<\frac{\log \lambda}{\log\frac\lambda{\sin\beta_0}}$. So, taking
$$
p_0:=\sup_{0<\lambda<\sqrt2-1}\frac{\log \lambda}{\log\frac\lambda{\sin\beta_0}},
$$
the part (\ref{i}) of the proposition is proved. By a numerical estimation via MatLab one can obtain $p_0\approx1.03946$.

\bigskip

Now let us demonstrate (\ref{ii}), which we will separate in different cases. Assume $p>p_0$.
\begin{enumerate}
\item Case R=1 and $\lambda=r<\sqrt2-1$. One has $\alpha_{p,r}<1$ and $T_{\mu,p}(1,r)\leq \sqrt\pi n \alpha_{p,r}^n$ by (\ref{cota.arriba.abajo1}).

\item Case $\sin\beta_0<R<\frac{\sqrt2}{1+\lambda}$. As before this implies that $2\sqrt{(1+\lambda)^{-2}-(1+\lambda)^{-4}}<R<1$ and $\lambda<\sqrt2-1$; for a fixed $\lambda$,
$$
\frac\partial{\partial R} \frac{R}{\sin\beta_0}=\frac{4R(1+\lambda)}{\left(4-R^2(1+\lambda)^2\right)^\frac32}>0,
$$
whenever $R$ is in the aforementioned range. This means that the upper and lower bounds in (\ref{cota.arriba.abajo}) are increasing with respect to $R$. Hence,
$$
T_{\mu,p}(R,r)\leq \sqrt\pi n\left(\frac{R\lambda^\frac{p-1}p}{\sin\beta_0}\right)^n \leq \sqrt\pi n \alpha_{p,\lambda}^n,
$$
with $\alpha_{p,\lambda}<1$, given that $\lambda<\sqrt2-1$.
\item Case $R<\frac{\sqrt2}{1+\lambda}$ and $R\leq \sin\beta_0$. By (\ref{cota.arriba.abajo}), we have that
    $$
    T_{\mu,p}(R,r) \leq \sqrt\pi n\left(\lambda^\frac{p-1}p\right)^n,
    $$
    and we are done because $\lambda^\frac{p-1}p<1$.
\item Case $R\geq \frac{\sqrt2}{1+\lambda}$. Then $\tilde B\cap B_1 \supset E_0 \cap B_1$, so $\mu(\tilde B)=|\tilde B \cap B_1| \geq |B_1\cap E_0| \geq \frac12|B_1|$. This implies that $T_{\mu,p}(R,r)\leq \frac12 \left(\lambda^\frac{p-1}p\right)^n$, for any $p>1$.
\end{enumerate}
So taking $\alpha=\max\{\alpha_{p,\lambda},(\sqrt2-1)^\frac{p-1}p\}$, the proof is complete.
\end{proof}

\subsection*{Acknowledgements}
This work was supported by MEC grant FPU-AP20050543 and DGU grant MTM2007-60952.

I would like to thank my advisor, F. Soria, for his help preparing this work and personal support.

\bigskip

\end{document}